\documentclass[11pt,letterpaper,reqno]{amsart}

\usepackage{amsmath,amssymb,amsthm,amsfonts}
\usepackage[T1]{fontenc}
\usepackage{doi}
\usepackage{hyperref}
\usepackage{color}
\hypersetup{pdfstartview={FitH},colorlinks=true,citecolor=blue,linkcolor=blue,urlcolor=blue}

\newtheorem{thm}{Theorem}[section]
\newtheorem{lem}[thm]{Lemma}
\newtheorem{prop}[thm]{Proposition}
\newtheorem{cor}[thm]{Corollary}
\newtheorem{prob}[thm]{Problem}
\newtheorem{conj}[thm]{Conjecture}
\theoremstyle{definition}
\newtheorem{remark}[thm]{Remark}

\numberwithin{equation}{section}

\newcommand{\Tr}{\operatorname{Tr}}
\newcommand{\rank}{\operatorname{rank}}
\newcommand{\Ree}{\operatorname{Re}}
\newcommand{\diag}{\operatorname{diag}}

\newcommand{\bh}{\mathbb{H}_n}
\newcommand{\bm}{\mathbb{M}_n}

\begin{document}
	
	\title[Weak majorization at orders three and four]{Weak majorization inequalities for the cubic and quartic coefficients of $e^{(A+B)t}$ versus $e^{At}e^{Bt}$}
	
	\author[T.~Zhang]{Teng Zhang}
	\address{School of Mathematics and Statistics, Xi'an Jiaotong University, Xi'an 710049, P.~R.~China}
	\email{teng.zhang@stu.xjtu.edu.cn}
	
	\subjclass[2020]{15A42, 15A60}
	\keywords{Hermitian matrix; singular values; weak majorization; Ky Fan principle; commutator}
	
	\begin{abstract}
		Let $A,B\in\bh$ and set $H=A+B$.
		For each integer $k\ge 1$ define
		\[
		Q_k:=\sum_{p=0}^k \binom{k}{p} A^pB^{k-p},
		\qquad
		R_k:=\Ree\,Q_k=\frac{Q_k+Q_k^*}{2}.
		\]
		Then $H^k=\left.\frac{d^k}{dt^k}e^{Ht}\right|_{t=0}$ and
		$Q_k=\left.\frac{d^k}{dt^k}(e^{At}e^{Bt})\right|_{t=0}$.
		We prove that, for $k=3,4$,
		\[
		\lambda(H^k)\prec_w \sigma(Q_k).
		\]
		Equivalently, the eigenvalues of the cubic and quartic Taylor coefficients of $e^{(A+B)t}$
		are weakly majorized by the singular values of the corresponding coefficients of the
		Golden--Thompson product $e^{At}e^{Bt}$.
		Our argument combines Ky Fan variational principles with explicit commutator identities
		for $R_k-H^k$ at orders $k=3,4$, reducing the problem to the positivity of certain double-commutator trace forms
		tested against Ky Fan maximizing projections.
		We also record a general sufficient condition for higher orders based on commutator
		decompositions.
	\end{abstract}
	
	\maketitle
	
	%========================================================
	\section{Introduction}
	%========================================================
	
	Let $\bm$ be the space of $n\times n$ complex matrices and let $\bh\subset\bm$ be the Hermitian matrices. We use $\preceq$ to denote the usual Loewner order on $\bh$.
	For $X\in\bh$ we write
	\[
	\lambda(X)=\bigl(\lambda_1(X),\dots,\lambda_n(X)\bigr)
	\]
	for the eigenvalues listed in nonincreasing order.
	For $Y\in\bm$ we write
	\[
	\sigma(Y)=\bigl(\sigma_1(Y),\dots,\sigma_n(Y)\bigr)
	\]
	for the singular values listed in nonincreasing order.
	
	\medskip
	
	Given $x,y\in\mathbb{R}^n$ (both arranged in nonincreasing order), we say that $x$ is \emph{weakly majorized} by $y$, written $x\prec_w y$, if
	\[
	\sum_{i=1}^r x_i \le \sum_{i=1}^r y_i,\qquad r=1,\dots,n.
	\]
	If equality holds also for $r=n$, we say that $x$ is \emph{majorized} by $y$, written $x\prec y$.
	
	If $x,y\in \mathbb{R}^n_+$, we say that $x$ is \emph{weakly log-majorized} by $y$, written $x\prec_{w\log} y$, if
	\[
	\prod_{i=1}^r x_i \le \prod_{i=1}^r y_i,\qquad r=1,\dots,n.
	\]
	If equality holds also for $r=n$, we say that $x$ is \emph{log-majorized} by $y$, written $x\prec_{\log} y$.
	For background on matrix majorization, see \cite{MOA11}.
	
	\medskip
	
	The Golden--Thompson inequality asserts that for all $A,B\in\bh$,
	\begin{equation}\label{eq:GT}
		\Tr e^{A+B} \le \Tr\!\left(e^{A}e^{B}\right),
	\end{equation}
	with equality if and only if $A$ and $B$ commute \cite{Len71,So92}.
	Thompson proved the stronger log-majorization refinement \cite{Tom71}
	\begin{equation}\label{eq:Thompson}
		\lambda(e^{A+B})\prec_{\log}\lambda(e^{A}e^{B}).
	\end{equation}
	Note that $e^{A}e^{B}$ is similar to the positive definite matrix $e^{B/2}e^{A}e^{B/2}$, hence
	its eigenvalues are positive and can be listed in nonincreasing order.
	See \cite{AH94,BPL04,Ber88,Bha97,CFKK82,FT14,Hia97,Hia16,HP93,Pet94} for further developments.
	
	\medskip
	
	Both \eqref{eq:GT} and \eqref{eq:Thompson} compare the \emph{full} exponentials.
	A natural bridge to more local information is provided by the Lie--Trotter product formula
	\cite{Kat78,Tro59}:
	\begin{equation}\label{eq:LT}
		e^{t(A+B)}=\lim_{n\to\infty}\Bigl(e^{\frac{t}{n}A}e^{\frac{t}{n}B}\Bigr)^n .
	\end{equation}
	In particular, the product $e^{tA}e^{tB}$ is the basic first-order splitting step
	approximating $e^{t(A+B)}$, which suggests asking whether one can compare, in a spectral sense,
	the Taylor coefficients of $e^{(A+B)t}$ and $e^{At}e^{Bt}$ at $t=0$.
	
	The Baker--Campbell--Hausdorff (BCH) formula
	\cite{Bak05,Cam97,Dyn47,Hal15,Hau06} makes precise how the
	noncommutativity of $A$ and $B$ enters: for small $t$,
	\[
	\log\!\bigl(e^{tA}e^{tB}\bigr)
	= t(A+B) + \frac{t^2}{2}[A,B]
	+ \frac{t^3}{12}\bigl([A,[A,B]]+[B,[B,A]]\bigr)+\cdots .
	\]
	Equivalently, expanding both sides at $t=0$ gives
	\[
	e^{(A+B)t}=I+t(A+B)+\frac{t^2}{2}(A^2+AB+BA+B^2)+\cdots,
	\]
	while
	\[
	e^{At}e^{Bt}=I+t(A+B)+\frac{t^2}{2}(A^2+2AB+B^2)+\cdots,
	\]
	so the first discrepancy already appears in the second derivative and is governed by the commutator:
	\[
	\frac{d^2}{dt^2}\Big|_{t=0}\!\left(e^{At}e^{Bt}-e^{(A+B)t}\right)= [A,B].
	\]
	Higher derivatives capture higher-order nested commutators predicted by BCH, and therefore provide
	a refined, coefficientwise view of how far the product $e^{At}e^{Bt}$ deviates from $e^{(A+B)t}$.
	This motivates studying spectral inequalities for the derivatives (or Taylor coefficients) at $t=0$,
	as a local counterpart to \eqref{eq:GT} and \eqref{eq:Thompson}.
	
	\medskip
	
	The results of this paper show that, at orders $k=3,4$, the deviation
	between the two expansions can be organized into explicit commutator terms that are
	nonnegative when tested against Ky Fan spectral projections, yielding the weak majorization
	inequalities $\lambda(H^k)\prec_w\sigma(Q_k)$.
	
	\medskip
	
	Fix $A,B\in\bh$ and set $H=A+B$.
	Expanding at $t=0$ gives
	\[
	e^{Ht}
	=
	I+Ht+\frac{H^2}{2}t^2+\frac{H^3}{3!}t^3+\frac{H^4}{4!}t^4+\cdots,
	\]
	and
	\[
	e^{At}e^{Bt}
	=
	I+Q_1t+\frac{Q_2}{2}t^2+\frac{Q_3}{3!}t^3+\frac{Q_4}{4!}t^4+\cdots,
	\qquad
	Q_k:=\sum_{p=0}^k\binom{k}{p}A^pB^{k-p}.
	\]
	(The formula for $Q_k$ follows from the Leibniz rule and $\frac{d^p}{dt^p}e^{At}\big|_{t=0}=A^p$.) Clearly, $H=Q_1$.
	Since $Q_2=A^2+2AB+B^2$, we have $H^2=\Ree\,Q_2$, hence by the Fan--Hoffman inequality (Lemma~\ref{lem:fan-hoffman}),
	\[
	\lambda(H^2)=\lambda(\Ree\,Q_2)\prec_w \sigma(Q_2).
	\]
	
	Motivated by this, we ask whether analogous comparisons hold at higher orders.
	
	\begin{prob}\label{prob:k}
		Let $A,B\in\bh$ and $k\ge 3$. Is it true that
		\[
		\lambda(H^k)\prec_w \sigma(Q_k)\,?
		\]
		Equivalently,
		\[
		\lambda\!\left(\frac{d^k}{dt^k}e^{(A+B)t}\Big|_{t=0}\right)
		\prec_w
		\sigma\!\left(\frac{d^k}{dt^k}(e^{At}e^{Bt})\Big|_{t=0}\right)?
		\]
	\end{prob}
	
	In this paper we prove the conjectured inequality for $k=3$ and $k=4$, and we give a
	general reduction and a sufficient condition suggesting an approach to general $k$.
	
	%========================================================
	\section{Preliminaries}
	%========================================================
	
	We use the following Ky Fan variational principle together with the Fan--Hoffman inequality.
	
	\begin{lem}[Ky Fan \cite{Fan49}]\label{lem:fan_eig}
		Let $X\in\bh$. Then for each $r=1,\dots,n$,
		\[
		\sum_{j=1}^r \lambda_j(X)
		=
		\max_{\substack{E=E^*=E^2\\ \rank(E)=r}} \Tr(EX).
		\]
	\end{lem}
	
	For $Y\in\bm$, let $\Ree\,Y=(Y+Y^*)/2$.
	
	\begin{lem}[Fan--Hoffman \cite{FH55}]\label{lem:fan-hoffman}
		Let $Y\in\bm$. Then for each $j=1,\dots,n$,
		\[
		\lambda_j(\Ree\,Y)\le \sigma_j(Y).
		\]
		Consequently, $\lambda(\Ree\,Y)\prec_w \sigma(Y)$.
	\end{lem}
	
	\begin{proof}The pointwise inequality $\lambda_j(\Ree\,Y)\le \sigma_j(Y)$ is classical and can also be found,
		for instance, in \cite[p.~73, Thm.~III.5.1]{Bha97}.
		Summing the inequalities for $j=1,\dots,r$ yields $\sum_{j=1}^r\lambda_j(\Ree\,Y)\le \sum_{j=1}^r\sigma_j(Y)$ for every $r$, i.e.\ $\lambda(\Ree\,Y)\prec_w\sigma(Y)$.
	\end{proof}
	
	We shall repeatedly use the following trace identity for commutators: for any $X,Y,Z\in\bm$,
	\begin{equation}\label{eq:trace-comm}
		\Tr\bigl(Z[X,Y]\bigr)=\Tr\bigl([Z,X]\,Y\bigr).
	\end{equation}
	It follows by expanding both sides and using cyclicity of the trace.
	
	By a “spectral projection associated with the $r$ largest eigenvalues” we mean
	an orthogonal projection onto an $r$-dimensional subspace spanned by eigenvectors
	corresponding to the $r$ largest eigenvalues (counting multiplicities).
	\begin{lem}\label{lem:double-comm-positive}
		Let $F\in\bh$ and let $X\in\bh$.
		Fix $r\in\{1,\dots,n\}$ and let $E$ be a spectral projection of $F$ associated with its $r$ largest eigenvalues.
		Then
		\[
		\Tr\bigl(E[X,[X,F]]\bigr)\ge 0.
		\]
		More precisely, if $F=\diag(f_1,\dots,f_n)$ with $f_1\ge\cdots\ge f_n$ and
		$E=\diag(\underbrace{1,\dots,1}_{r},\underbrace{0,\dots,0}_{n-r})$, and if $X=(x_{ij})$ in this basis, then
		\begin{equation}\label{eq:explicit-double-comm}
			\Tr\bigl(E[X,[X,F]]\bigr)
			=
			2\sum_{1\le i\le r<j\le n}(f_i-f_j)\,|x_{ij}|^2.
		\end{equation}
	\end{lem}
	
	\begin{proof}
		By \eqref{eq:trace-comm} with $Z=E$ and $Y=[X,F]$,
		\[
		\Tr\bigl(E[X,[X,F]]\bigr)=\Tr\bigl([E,X]\,[X,F]\bigr).
		\]
		Diagonalize $F$ and choose an orthonormal basis so that
		$F=\diag(f_1,\dots,f_n)$ with $f_1\ge\cdots\ge f_n$ and
		$E=\diag(\underbrace{1,\dots,1}_{r},\underbrace{0,\dots,0}_{n-r})$.
		Write $X=(x_{ij})$ in this basis. Then
		\[
		[X,F]_{ij}=x_{ij}(f_j-f_i),
		\qquad
		[E,X]_{ij}=x_{ij}(e_i-e_j),
		\quad e_i\in\{0,1\}.
		\]
		Using $\Tr(AB)=\sum_{i,j}A_{ij}B_{ji}$ and $X=X^*$ (so $x_{ji}=\overline{x_{ij}}$), we obtain
		\[
		\Tr\bigl([E,X]\,[X,F]\bigr)
		=
		\sum_{i,j}(e_i-e_j)(f_i-f_j)\,|x_{ij}|^2.
		\]
		The only nonzero contributions are those with exactly one of $i,j$ in $\{1,\dots,r\}$.
		Pairing the terms $(i,j)$ and $(j,i)$ yields \eqref{eq:explicit-double-comm}.
		Since $f_i\ge f_j$ for $i\le r<j$, every summand is nonnegative.
	\end{proof}
	
	%========================================================
	\section{The cubic coefficient ($k=3$)}
	%========================================================
	
	Throughout, set
	\[
	H:=A+B,\qquad X:=A-B.
	\]
	Note that $[X,H]=2[A,B]$.
	
	\begin{lem}\label{lem:D3}
		Let $A,B\in\bh$ and define
		\[
		Q_3:=A^3+3A^2B+3AB^2+B^3,
		\qquad
		R_3:=\Ree\,Q_3.
		\]
		Then
		\begin{equation}\label{eq:D3}
			R_3-H^3=\frac14[X,[X,H]].
		\end{equation}
	\end{lem}
	
	\begin{proof}
		A direct expansion gives
		\[
		R_3=A^3+B^3+\frac32\bigl(A^2B+BA^2+AB^2+B^2A\bigr),
		\]
		and
		\[
		H^3=A^3+B^3+\bigl(A^2B+ABA+BA^2\bigr)+\bigl(AB^2+BAB+B^2A\bigr).
		\]
		Subtracting yields
		\[
		R_3-H^3
		=
		\frac12\bigl(A^2B+BA^2+AB^2+B^2A\bigr)-(ABA+BAB).
		\]
		On the other hand,
		\[
		[X,[X,H]]
		=
		[A-B,[A-B,A+B]]
		=
		2[A-B,[A,B]]
		=
		2\bigl(A^2B+BA^2+AB^2+B^2A-2ABA-2BAB\bigr),
		\]
		hence \eqref{eq:D3}.
	\end{proof}
	
	\begin{thm}\label{thm:k3-eig}
		Let $A,B\in\bh$ and let $H=A+B$, $Q_3$ as above. Then
		\[
		\lambda(H^3)\prec \lambda(\Ree\,Q_3).
		\]
		Consequently,
		\[
		\lambda(H^3)\prec_w \sigma(Q_3).
		\]
	\end{thm}
	
	\begin{proof}First, by Lemma~\ref{lem:D3} and $\Tr([U,V])=0$ for any $U,V\in\bm$, we have
		\begin{equation}\label{eq:tr_H3=R3}
			\Tr(H^3)=\Tr(R_3)=\Tr(\Ree\,Q_3).
		\end{equation}
		
Fix $r\in\{1,\dots,n\}$.
		Let $E$ be a spectral projection of $H$ associated with its $r$ largest eigenvalues.
		Since $t\mapsto t^3$ is increasing on $\mathbb{R}$ and $H^3$ is a polynomial in $H$, the same $E$ is also a spectral projection of $H^3$ associated with its $r$ largest eigenvalues.
		By Lemma~\ref{lem:fan_eig},
		\[
		\sum_{j=1}^r\lambda_j(H^3)=\Tr(EH^3),
		\qquad
		\sum_{j=1}^r\lambda_j(\Ree\,Q_3)\ge \Tr(E\,\Ree\,Q_3).
		\]
		Thus it suffices to show $\Tr(E\,\Ree\,Q_3)\ge \Tr(EH^3)$, i.e.
		\[
		\Tr\bigl(E(\Ree\,Q_3-H^3)\bigr)\ge 0.
		\]
		By Lemma~\ref{lem:D3},
		\[
		\Tr\bigl(E(\Ree\,Q_3-H^3)\bigr)=\frac14\Tr\bigl(E[X,[X,H]]\bigr)\ge 0,
		\]
		where the last inequality is Lemma~\ref{lem:double-comm-positive} with $F=H$. 
		Therefore $\lambda(H^3)\prec_w\lambda(\Ree\,Q_3)$. Together with \eqref{eq:tr_H3=Q3}, we have $\lambda(H^3)\prec \lambda(\Ree\,Q_3)$.

		Finally, Lemma~\ref{lem:fan-hoffman}  gives $\lambda(\Ree\,Q_3)\prec_w\sigma(Q_3)$, and transitivity of $\prec_w$ yields $\lambda(H^3)\prec_w\sigma(Q_3)$.
	\end{proof}
	
	\begin{cor}\label{cor:k3-derivative}
		For $A,B\in\bh$,
		\[
		\lambda\!\left(\frac{d^3}{dt^3}e^{(A+B)t}\Big|_{t=0}\right)
		\prec_w
		\sigma\!\left(\frac{d^3}{dt^3}(e^{At}e^{Bt})\Big|_{t=0}\right).
		\]
	\end{cor}
	
	\begin{proof}
		This is Theorem~\ref{thm:k3-eig} with $H^3=\frac{d^3}{dt^3}e^{Ht}\big|_{t=0}$ and $Q_3=\frac{d^3}{dt^3}(e^{At}e^{Bt})\big|_{t=0}$.
	\end{proof}
	
	%========================================================
	\section{The quartic coefficient ($k=4$)}
	%========================================================
	
	\begin{lem}\label{lem:D4}
		Let $A,B\in\bh$ and define
		\[
		Q_4:=A^4+4A^3B+6A^2B^2+4AB^3+B^4,
		\qquad
		R_4:=\Ree\,Q_4.
		\]
		With $H=A+B$ and $X=A-B$ we have the identity
		\begin{equation}\label{eq:D4-}
			R_4-H^4=\frac12[X,[X,H^2]]-\frac14[X,H]^2.
		\end{equation}
	\end{lem}
	\begin{proof}
		Write $A=(H+X)/2$ and $B=(H-X)/2$.
		Both sides of \eqref{eq:D4-} are noncommutative polynomials in $H$ and $X$ with scalar coefficients.
		Thus it suffices to expand both sides in the free algebra generated by $H$ and $X$ and compare coefficients of all words of length four.
		A straightforward (but routine) expansion verifies \eqref{eq:D4-}, see Appendix \ref{app:D4}.
	\end{proof}
	
	\begin{lem}\label{lem:skew-square}
		If $K\in\bm$ satisfies $K^*=-K$, then $-K^2=K^*K\succeq 0$.
	\end{lem}
	
	\begin{proof}
		Since $K^*=-K$, we have $K^*K=(-K)K=-K^2$. As $K^*K$ is positive semidefinite, so is $-K^2$.
	\end{proof}
	
	\begin{thm}\label{thm:k4-eig}
		Let $A,B\in\bh$ and let $H=A+B$, $Q_4$ as above. Then
		\[
		\lambda(H^4)\prec_w \lambda(\Ree\,Q_4).
		\]
		Consequently,
		\[
		\lambda(H^4)\prec_w \sigma(Q_4).
		\]
	\end{thm}
	
	\begin{proof}
		Fix $r\in\{1,\dots,n\}$.
		Since $H^2\succeq 0$ and $H^4=(H^2)^2$, the map $t\mapsto t^2$ is increasing on $[0,\infty)$.
		Hence we may choose $E$ to be a spectral projection of $H^2$ associated with its $r$ largest eigenvalues; the same $E$ is then also a spectral projection of $H^4$ associated with its $r$ largest eigenvalues.
		By Lemma~\ref{lem:fan_eig},
		\[
		\sum_{j=1}^r\lambda_j(H^4)=\Tr(EH^4),
		\qquad
		\sum_{j=1}^r\lambda_j(\Ree\,Q_4)\ge \Tr(E\,\Ree\,Q_4).
		\]
		Hence it suffices to prove $\Tr(E(\Ree\,Q_4-H^4))\ge 0$.
		
		Using Lemma~\ref{lem:D4},
		\[
		\Tr\bigl(E(\Ree\,Q_4-H^4)\bigr)
		=
		\frac12\Tr\bigl(E[X,[X,H^2]]\bigr)-\frac14\Tr\bigl(E[X,H]^2\bigr).
		\]
		The first term is nonnegative by Lemma~\ref{lem:double-comm-positive} applied with $F=H^2$.
		For the second term, note that $X,H\in\bh$ implies $[X,H]^*=-[X,H]$ (skew-Hermitian), hence by Lemma~\ref{lem:skew-square} we have $-[X,H]^2\succeq 0$. Therefore,
		\[
		-\frac14\Tr\bigl(E[X,H]^2\bigr)
		=
		\frac14\Tr\bigl(E(-[X,H]^2)\bigr)\ge 0,
		\]
 since $E\succeq 0$ and $-[X,H]^2\succeq 0$.

		Thus $\Tr(E(\Ree\,Q_4-H^4))\ge 0$, proving $\lambda(H^4)\prec_w\lambda(\Ree\,Q_4)$.
		Finally, Lemma~\ref{lem:fan-hoffman} implies $\lambda(\Ree\,Q_4)\prec_w\sigma(Q_4)$, hence $\lambda(H^4)\prec_w\sigma(Q_4)$.
	\end{proof}
	
	\begin{cor}\label{cor:k4-derivative}
		For $A,B\in\bh$,
		\[
		\lambda\!\left(\frac{d^4}{dt^4}e^{(A+B)t}\Big|_{t=0}\right)
		\prec_w
		\sigma\!\left(\frac{d^4}{dt^4}(e^{At}e^{Bt})\Big|_{t=0}\right).
		\]
	\end{cor}
	
	\begin{proof}
		This is Theorem~\ref{thm:k4-eig} with $H^4=\frac{d^4}{dt^4}e^{Ht}\big|_{t=0}$ and $Q_4=\frac{d^4}{dt^4}(e^{At}e^{Bt})\big|_{t=0}$.
	\end{proof}
	
	%========================================================
	\section{Towards general $k$}
	%========================================================
	
	Throughout this section we again set
	\[
	H:=A+B,\qquad X:=A-B.
	\]
	For $k\ge 1$ define as before
	\[
	Q_k=\sum_{p=0}^k \binom{k}{p}A^pB^{k-p},
	\qquad
	R_k=\Ree\,Q_k,
	\qquad
	D_k:=R_k-H^k.
	\]
By Lemma~\ref{lem:fan-hoffman}, $\lambda(R_k)\prec_w\sigma(Q_k)$ always holds.
Hence it suffices to prove the Hermitian comparison $\lambda(H^k)\prec_w\lambda(R_k)$.

	\begin{prop}[Ky Fan reduction]\label{prop:reduction}
		Fix $k\ge 1$ and, for each $r=1,\dots,n$, let $E_{k,r}$ be a spectral projection of $H^k$
		associated with its $r$ largest eigenvalues.
		If
		\[
		\Tr(E_{k,r}D_k)\ge 0,\qquad r=1,\dots,n,
		\]
		then $\lambda(H^k)\prec_w\lambda(R_k)$.
	\end{prop}
	
	\begin{proof}
		By Lemma~\ref{lem:fan_eig}, for each $r$,
		\[
		\sum_{j=1}^r\lambda_j(H^k)=\max_{\substack{E=E^*=E^2\\ \rank(E)=r}}\Tr(EH^k)=\Tr(E_{k,r}H^k).
		\]
		On the other hand,
		\[
		\sum_{j=1}^r\lambda_j(R_k)=\max_{\substack{E=E^*=E^2\\ \rank(E)=r}}\Tr(ER_k)\ge \Tr(E_{k,r}R_k).
		\]
		Therefore, if $\Tr(E_{k,r}D_k)\ge 0$, then
		\[
		\sum_{j=1}^r\lambda_j(H^k)=\Tr(E_{k,r}H^k)\le \Tr(E_{k,r}R_k)\le \sum_{j=1}^r\lambda_j(R_k)
		\quad (r=1,\dots,n),
		\]
		which is exactly $\lambda(H^k)\prec_w\lambda(R_k)$.
	\end{proof}
	
	\begin{remark}\label{rem:odd-even}
		Let $k\ge 1$ and $r\in\{1,\dots,n\}$.
		Because $H^k$ is a polynomial in $H$, the matrices $H$ and $H^k$ commute and are simultaneously diagonalizable.
		\begin{itemize}
			\item If $k$ is odd, then $t\mapsto t^k$ is increasing on $\mathbb{R}$, so the order of eigenvalues is preserved under $t\mapsto t^k$. Hence any spectral projection of $H$ associated with its $r$ largest eigenvalues is also a spectral projection of $H^k$ associated with its $r$ largest eigenvalues.
			\item If $k$ is even, then $H^k=(H^2)^{k/2}$ with $H^2\succeq 0$, and $t\mapsto t^{k/2}$ is increasing on $[0,\infty)$. Hence any spectral projection of $H^2$ associated with its $r$ largest eigenvalues is also a spectral projection of $H^k$ associated with its $r$ largest eigenvalues.
		\end{itemize}
		In particular, for each $r$ we may (and will) choose $E_{k,r}$ so that it is a spectral projection of $H$ when $k$ is odd, and a spectral projection of $H^2$ when $k$ is even.
	\end{remark}
	
	\begin{thm}[A sufficient condition via commutator decompositions]\label{thm:sufficient-decomp}
		Fix $k\ge 1$ and define $D_k=R_k-H^k$ as above.
		Suppose that there exist integers $m,\ell\ge 0$, scalars $c_1,\dots,c_m\ge 0$, Hermitian matrices $F_1(H),\dots,F_m(H)\in\bh$, and matrices $W_1,\dots,W_\ell\in\bm$ such that
		\begin{equation}\label{eq:sufficient-decomp}
			D_k
			=
			\sum_{j=1}^m c_j\,[X,[X,F_j(H)]] \;+\; \sum_{q=1}^\ell W_q^*W_q.
		\end{equation}
		Assume moreover that each $F_j(H)$ is an increasing function of the relevant operator in Remark~\ref{rem:odd-even}, namely:
		\begin{itemize}
			\item if $k$ is odd, $F_j(H)=f_j(H)$ for some increasing scalar function $f_j:\mathbb{R}\to\mathbb{R}$;
			\item if $k$ is even, $F_j(H)=f_j(H^2)$ for some increasing scalar function $f_j:[0,\infty)\to\mathbb{R}$.
		\end{itemize}
		Then $\lambda(H^k)\prec_w\lambda(R_k)$ and hence $\lambda(H^k)\prec_w\sigma(Q_k)$.
	\end{thm}
	
	\begin{proof}
		Fix $r$ and choose $E_{k,r}$ as in Remark~\ref{rem:odd-even}.
		By the monotonicity assumption on each $F_j(H)$ and simultaneous diagonalizability, the same $E_{k,r}$ is also a spectral projection of $F_j(H)$ corresponding to its $r$ largest eigenvalues.
		Therefore Lemma~\ref{lem:double-comm-positive} gives
		\[
		\Tr\bigl(E_{k,r}[X,[X,F_j(H)]]\bigr)\ge 0
		\quad\text{for all }j.
		\]
		Also, $\Tr(E_{k,r}W_q^*W_q)\ge 0$ for all $q$ because $E_{k,r}\succeq 0$ and $W_q^*W_q\succeq 0$.
		
		Taking the trace of \eqref{eq:sufficient-decomp} against $E_{k,r}$ yields
		$\Tr(E_{k,r}D_k)\ge 0$.
		By Proposition~\ref{prop:reduction}, $\lambda(H^k)\prec_w\lambda(R_k)$.
		Finally, Lemma~\ref{lem:fan-hoffman} gives $\lambda(R_k)\prec_w\sigma(Q_k)$, and transitivity yields $\lambda(H^k)\prec_w\sigma(Q_k)$.
	\end{proof}
	
	\begin{remark}\label{rem:k34-fit}
		For $k=3$, Lemma~\ref{lem:D3} gives $D_3=\frac14[X,[X,H]]$, which is of the form \eqref{eq:sufficient-decomp} with $m=1$, $c_1=\frac14$, $F_1(H)=H$, and $\ell=0$.
		For $k=4$, Lemma~\ref{lem:D4} gives $D_4=\frac12[X,[X,H^2]]-\frac14[X,H]^2$, which is of the form \eqref{eq:sufficient-decomp} with $m=1$, $c_1=\frac12$, $F_1(H)=H^2$, and $\ell=1$ by taking $W_1=\frac12[X,H]$ and using $W_1^*W_1=-\frac14[X,H]^2$ (Lemma~\ref{lem:skew-square}).
	\end{remark}

	\begin{conj}\label{conj:general}
		For $A,B\in\bh$ and every $k\ge 5$,
		\[
	\lambda(H^k)\prec_w\lambda(R_k).
		\]
	\end{conj}

	\begin{remark}[An explicit (but more intricate) commutator expansion at $k=5$]\label{rem:D5}
		For completeness we record one convenient identity for $D_5=R_5-H^5$ in terms of $H$ and $X$.
		Let $\{U,V\}=UV+VU$ be the anticommutator and $\operatorname{ad}_X(Y)=[X,Y]$.
		A direct expansion yields
		\[
		\begin{aligned}
			D_5
			&=
			\frac{1}{16}\operatorname{ad}_X^{4}(H)
			+\frac{7}{16}\operatorname{ad}_X^{2}(H^3)
			+\frac{9}{32}\{H,\operatorname{ad}_X^{2}(H^2)\}
			-\frac{1}{32}\{H^2,\operatorname{ad}_X^{2}(H)\}
			+\frac18\,H\,\operatorname{ad}_X^{2}(H)\,H.
		\end{aligned}
		\]
		Unlike $k=3,4$, this formula contains higher iterated commutators and mixed $H$--$\operatorname{ad}_X^2(\cdot)$ terms, and we do not currently know how to organize it into a decomposition of the form \eqref{eq:sufficient-decomp} in full generality.
	\end{remark}
	
	%========================================================
	\section*{Acknowledgments}
	Teng Zhang is supported by the China Scholarship Council, the Young Elite Scientists Sponsorship Program for PhD Students (China Association for Science and Technology), and the Fundamental Research Funds for the Central Universities at Xi'an Jiaotong University (Grant No.~xzy022024045).
	He is grateful to his advisor, Minghua Lin, for bringing this topic to his attention.
	
	%========================================================
	% Bibliography
	%========================================================
	
	\appendix
	\section{Derivation of the identity \texorpdfstring{\eqref{eq:D4}}{(D4)}}\label{app:D4}
	
	In this appendix we provide a detailed expansion proof of Lemma~\ref{lem:D4}.
	
	\begin{lem}[Lemma~\ref{lem:D4}]\label{lem:D4-app}
		Let $A,B\in\bh$ and define
		\[
		Q_4:=A^4+4A^3B+6A^2B^2+4AB^3+B^4,
		\qquad
		R_4:=\Ree\,Q_4.
		\]
		With $H=A+B$ and $X=A-B$ we have the identity
		\begin{equation}\label{eq:D4}
			R_4-H^4=\frac12[X,[X,H^2]]-\frac14[X,H]^2.
		\end{equation}
	\end{lem}
	
	\begin{proof}
		Throughout we use $H=A+B$ and $X=A-B$.
		Since $A,B$ are Hermitian, we have
		\[
		Q_4^*=A^4+4BA^3+6B^2A^2+4B^3A+B^4,
		\]
		hence
		\begin{equation}\label{eq:R4-expanded}
			R_4=\Ree\,Q_4=\frac12(Q_4+Q_4^*)
			=A^4+B^4+2A^3B+3A^2B^2+2AB^3+2BA^3+3B^2A^2+2B^3A .
		\end{equation}
		
		\medskip\noindent
		\emph{Step 1: Expansion of $R_4-H^4$.}
		Write
		\[
		H^4=(A+B)^4=\sum_{w\in\{A,B\}^4} w ,
		\]
		i.e., the sum of all $16$ words of length $4$ in the alphabet $\{A,B\}$.
		Comparing with \eqref{eq:R4-expanded}, we see that $R_4$ contains only the eight block words
		\[
		A^4,\ B^4,\ A^3B,\ BA^3,\ A^2B^2,\ B^2A^2,\ AB^3,\ B^3A,
		\]
		with coefficients $1,1,2,2,3,3,2,2$, respectively, and all other words have coefficient $0$.
		Therefore,
		\begin{align}
			R_4-H^4
			&=\bigl(A^3B+BA^3+2A^2B^2+2B^2A^2+AB^3+B^3A\bigr)\label{eq:LHS}\\
			&\quad-\bigl(A^2BA+ABA^2+ABAB+AB^2A+BA^2B+BABA+BAB^2+B^2AB\bigr). \nonumber
		\end{align}
		
		\medskip\noindent
		\emph{Step 2: Expansion of the commutator side.}
		First compute
		\[
		[X,H]=(A-B)(A+B)-(A+B)(A-B)=2(AB-BA),
		\]
		hence
		\begin{equation}\label{eq:comm-sq}
			-\frac14[X,H]^2=-(AB-BA)^2
			= -ABAB+AB^2A+BA^2B-BABA .
		\end{equation}
		
		Next note that
		\[
		[X,[X,H^2]]=X(XH^2-H^2X)-(XH^2-H^2X)X
		=X^2H^2+H^2X^2-2XH^2X,
		\]
		so
		\begin{equation}\label{eq:double-comm-form}
			\frac12[X,[X,H^2]]
			=\frac12(X^2H^2+H^2X^2)-XH^2X.
		\end{equation}
		We now expand each term.
		
		Since
		\[
		H^2=(A+B)^2=A^2+AB+BA+B^2,\qquad
		X^2=(A-B)^2=A^2-AB-BA+B^2,
		\]
		a direct multiplication gives
		\begin{align*}
			X^2H^2
			&=(A^2-AB-BA+B^2)(A^2+AB+BA+B^2)\\
			&=A^4+A^3B+A^2BA+A^2B^2
			-ABA^2-ABAB-AB^2A-AB^3\\
			&\quad-BA^3-BA^2B-BABA-BAB^2
			+B^2A^2+B^2AB+B^3A+B^4 .
		\end{align*}
		Taking adjoints yields $H^2X^2=(X^2H^2)^*$, and hence
		\begin{align}
			\frac12(X^2H^2+H^2X^2)
			&=A^4+B^4+A^2B^2+B^2A^2 \label{eq:avg}\\
			&\quad-\bigl(ABAB+AB^2A+BA^2B+BABA\bigr).\nonumber
		\end{align}
		
		Moreover,
		\begin{align}
			XH^2X
			&=(A-B)(A^2+AB+BA+B^2)(A-B)\nonumber\\
			&=A^4-A^3B+A^2BA-A^2B^2+ABA^2-ABAB+AB^2A-AB^3 \label{eq:XH2X}\\
			&\quad-BA^3+BA^2B-BABA+BAB^2-B^2A^2+B^2AB-B^3A+B^4.\nonumber
		\end{align}
		
		Subtracting \eqref{eq:XH2X} from \eqref{eq:avg} and using \eqref{eq:double-comm-form}, we obtain
		\begin{align}
			\frac12[X,[X,H^2]]
			&=\bigl(A^3B+BA^3+2A^2B^2+2B^2A^2+AB^3+B^3A\bigr)\label{eq:double-comm-expanded}\\
			&\quad-\bigl(A^2BA+ABA^2+2AB^2A+2BA^2B+BAB^2+B^2AB\bigr).\nonumber
		\end{align}
		
		Finally, combining \eqref{eq:double-comm-expanded} with \eqref{eq:comm-sq} gives
		\begin{align*}
			\frac12[X,[X,H^2]]-\frac14[X,H]^2
			&=\bigl(A^3B+BA^3+2A^2B^2+2B^2A^2+AB^3+B^3A\bigr)\\
			&\quad-\bigl(A^2BA+ABA^2+ABAB+AB^2A+BA^2B+BABA+BAB^2+B^2AB\bigr),
		\end{align*}
		which coincides exactly with \eqref{eq:LHS}. This proves \eqref{eq:D4}.
	\end{proof}
	
\end{document}